\documentclass[12pt]{article}

\usepackage{amssymb}
\usepackage{amsthm}
\usepackage{amsmath}

\usepackage{graphicx}
\usepackage[figuresright]{rotating}

\usepackage{float}
\usepackage{makeidx}
\usepackage{amsfonts}
\usepackage[margin=2.5cm]{geometry}
\usepackage{enumerate}
\theoremstyle{plain}
\newtheorem{theorem}{Theorem}[section]

\newtheorem{lemma}[theorem]{Lemma}
\newtheorem{corollary}[theorem]{Corollary}
\newtheorem{conjecture}[theorem]{Conjecture}
\newtheorem{proposition}[theorem]{Proposition}

\newtheorem{problem}[theorem]{Problem}
\theoremstyle{definition} 
\newtheorem{claim}[theorem]{Claim}

\newcommand{\defi}{\mathrm{def}}

\numberwithin{equation}{section}

\title{Local measures of interval edge-uncolorability}

\author{Carl Johan Casselgren,Petros A. Petrosyan}

\author{
{\sl Carl Johan Casselgren}\thanks{{\it E-mail address:} 
carl.johan.casselgren@liu.se}\\ 
Department of Mathematics \\
Link\"oping University \\ 
SE-581 83 Link\"oping, Sweden
\and
{\sl Petros A. Petrosyan}\thanks{{\it E-mail address:} 
petros\_petrosyan@ysu.am} \\
Department of Informatics \\ 
and Applied Mathematics,\\
Yerevan State University \\ 
0025, Armenia
}

\begin{document}

\maketitle

\begin{abstract}
An interval edge coloring of a graph is a proper edge coloring by integers
such that the colors on the edges incident with any vertex form an interval
of integers. Not all graphs are interval colorable; a simple counterexample is $K_3$.

The {\em (interval coloring) deficiency} of a graph $G$ is the minimum number of pendant 
edges whose addition to $G$ yields a graph with an interval edge coloring.
In this paper, we introduce and study further measures of how far from being interval
colorable a graph is. 
The {\em local deficiency} of a graph $G$ is the smallest
number of pendant edges that needs to be added at every vertex of $G$ in order to obtain
a graph with an interval edge coloring; we can think of the colors of these
added edges as ''locally missing'' at a vertex.
We also study a weaker version of this notion, the {\em weak local deficiency},
which informally is the size of a largest set of consecutive integers ''locally missing'' at
a vertex in a proper edge coloring of $G$ minimizing this size.

We compare weak local deficiency, local deficiency, and deficiency, and show
that the difference can be arbitrarily large in both cases.  Moreover, 
we give concrete examples of graphs whose weak local deficiency
(and thus local deficiency) grows
with the number of vertices as well as with the maximum degree.
We also prove some constructive results on graphs with small weak local deficiency.
In particular, all complete multipartite graphs have weak local deficiency
at most $2$, and many complete multipartite graphs have weak local deficiency at most $1$.
Moreover, bipartite graphs with maximum degree at most $6$, and Eulerian bipartite
graphs with maximum degree at most $8$ both have weak local deficiency at most $1$.
We conclude the paper by pointing to several open questions for further research.


\end{abstract}

	
\bigskip

\section{Introduction}\

	Edge coloring can be used for modeling a wide variety of problems arising in
	computer science and operations research such as
	scheduling, frequency assignment, register allocations, etc.
	Many applications involve extra constraints. It may for instance be the case
	that we want to avoid interruptions in a scheduling problem; this constraint naturally
	leads to the notion of an {\em interval coloring}.
	
An {\em interval coloring} of a graph  is a proper edge coloring  by integers such that the colors on the edges incident to any vertex  form an interval of integers; this notion was introduced by Asratian and Kamalian  
\cite{AsrKam} (available in English as 
\cite{{AsrKamJCTB}}), motivated by the problem of finding compact school timetables, that is, timetables such that the lectures of each teacher and each class are scheduled at consecutive periods.

Not every graph has an interval coloring, since a graph
$G$ with an interval coloring must have
a proper $\Delta(G)$-edge coloring \cite{AsrKam}, where
$\Delta(G)$ denotes the maximum degree of $G$.
Sevastjanov  \cite{Sevastjanov} proved that  determining whether
a bipartite graph has an interval coloring is $\mathcal{NP}$-complete.
Nevertheless, trees \cite{Hansen, Kampreprint},
regular and complete bipartite graphs \cite{AsrKam,Hansen, Kampreprint},
grids \cite{GiaroKubale1}, doubly convex bipartite graphs \cite{AsrDenHag}, subcubic bipartite graphs 
\cite{Hansen},
and simple outerplanar bipartite graphs \cite{GiaroKubale2}
all have interval colorings. Moreover, Class 1 graphs of maximum degree at most $3$ are interval colorable \cite{AsratianCasselgrenPetrosyan}. Further results on interval colorings appear in e.g. \cite{Kubale}.

A well-known conjecture \cite{JensenToft,STSCHF} suggests that all $(a,b)$-biregular
graphs have interval colorings, where a bipartite graph is {\em $(a,b)$-biregular}
if all vertices in one part have degree $a$ and all vertices in the other part have degree $b$.
This seemingly very difficult conjecture is still wide open, although it is known that
all $(2,b)$-biregular \cite{Hansen, HansonLotenToft, KamMir, unpublished}
as well as $(3,6)$-biregular graphs \cite{CarlJToft} have interval colorings.
The first unsolved case $(a,b)=(3,4)$ have been considered by several authors
\cite{AsratianCasselgrenVandenWest, Pyatkin, YangLi}, but remains open.

In the context of interval coloring, the {\em deficiency} $\text{def}(G)$ of a graph $G$ is the minimum number
of pendant edges that has to be added to $G$ to yield an interval colorable graph. This notion 
was first considered by Giaro et al.~\cite{GiaroKubaleMalaf1, GiaroKubaleMalaf2}. They determined
the deficiency of some basic families of graphs, and proved that there are families of graphs 
whose deficiency approaches the number of vertices. Based on this evidence, Bouchard, Hertz and Desaulniers \cite{BouchardHertz} conjectured that $\defi(G) \leq |V(G)|$ for every graph $G$. Deficiency of graphs have been studied in several subsequent papers such as \cite{Schwartz, B-OD-BHal, BorowieckaDrgas, Hrant, BouchardHertz, PetrosHrant2}.

In this paper, we introduce and study further measures of how far from being interval colorable a graph is. 
The {\em local deficiency} of a graph $G$, denoted by $\mathrm{def}_{loc}(G)$, is the smallest number of pendant edges that needs to be added at every vertex of $G$ in order to obtain
a graph with an interval edge coloring; we can think of the colors of these
added edges as ``locally missing'' at a vertex in an interval coloring of the resulting graph.
We also study a weaker version of this notion, the {\em weak local deficiency}.

The {\em weak local deficiency} of a set $A$ of integers, denoted by $\mathrm{def}_{wloc}(A)$, is the largest set $I$ of consecutive integers that is necessary to add to $A$ to obtain an interval of integers; for example, the weak local deficiency
of $\{1,2,7,8,12,15\}$ is $4$, because $|\{3,4,5,6\}|=4$.
The {\em weak local deficiency $\mathrm{def}_{wloc}(G,f)$ of an edge coloring} $f$ of $G$ 
is defined as $\mathrm{def}_{wloc}(G,f)=\max_{v\in V(G)} \mathrm{def}_{wloc}(f(v))$, where
$f(v)$ denotes the set of colors present on edges incident with $v$.
The {\em weak local deficiency $\mathrm{def}_{wloc}(G)$ of $G$} is the minimum of
$\mathrm{def}_{wloc}(G,f)$ taken over all proper edge colorings $f$ of $G$.

Thus,
informally we can think of the weak local deficiency as the smallest size of a 
largest set of consecutive integers ''missing'' at
a vertex in a proper edge coloring of $G$, where {\em smallest} refers to the choice of the edge coloring of $G$.
So where local deficiency informally
counts the maximum number of ''missing colors'' at a vertex, the weak local deficiency
counts the maximum number of missing consecutive colors, i.e. the maximum size of a ``gap''
in the set of colors present at a vertex.

Edge colorings with small local deficiency, say local deficiency $k$, has a natural interpretation as schedules
where every involved party has at most $k$ interruptions, 
and is thus interesting in applications where compact schedules
are desirable. On the other hand, edge colorings with weak local deficiency $k$
correspond to schedules where no involved party has an interruption of length
greater than $k$ consecutive time periods, and is thus relevant for situations
where  interruptions may be acceptable but we need to bound the maximum
idle time.

Many graphs with large deficiency satisfy that only one edge needs to be added at every vertex
to yield a graph with an interval coloring; this holds for  e.g. regular graphs, both for the dense case
(complete graphs of odd order), as well as for the sparse case (disjoint union of odd cycles).
Thus, there are graphs with large deficiency which {\em locally} are ``almost'' interval colorable;
such graphs have been studied in the context of {\em near-interval colorings},
see e.g.  \cite{CarlJToft, CasselgrenPetrosyan, PetAraBagh, PetrosHranTigran}.
Local deficiency generalizes this property in a natural way.

Trivially, we have that $\mathrm{def}_{wloc}(G) \leq  \mathrm{def}_{loc}(G) \leq \mathrm{def}(G)$ for every graph $G$,
and $$\mathrm{def}_{wloc}(G) =  \mathrm{def}_{loc}(G) = \mathrm{def}(G)=0$$ if and only if $G$ is interval colorable.

By the example of regular graphs, the deficiency of a graph
may be arbitrary much larger than the local deficiency. We shall show that the same
holds for weak local deficiency compared to local deficiency.

Clearly, $\mathrm{def}_{loc}(G) \leq \Delta(G)- \delta(G)+1$ for any graph $G$,
where  $\delta(G)$ denotes the minimum degree in $G$.
Here, we present concrete families of graphs whose local deficiency grows linearly both 
with the maximum degree and
with the number of vertices. We also give examples where the weak local deficiency grows with the maximum degree
and with the number of vertices.

Khachatrian \cite{Hrant} established an upper bound on the deficiency of an outerplanar graph.
His proof  yields an upper bound on the local deficiency of a graph, and, moreover, 
it also implies that the weak local deficiency of an outerplanar graph is at most $1$.

In the present paper, we find further families of graphs with small local deficiency
and weak local deficiency, with a particular focus on the case when these invariants equal $1$.
In particular, we prove the following constructive results:
\begin{itemize}
	
	\item  Complete tripartite graphs have weak local deficiency at most $1$, general complete multipartite graphs have weak local deficiency at most $2$,
	and such graphs with a unique part of largest size have weak local deficiency at most $1$.

	\item
	Bipartite graphs with maximum degree at most $6$, 
	and
	Eulerian bipartite graphs with maximum degree at most $8$ both have 
	weak local deficiency at most $1$.

	\item Graphs with maximum degree at most $5$, where no two vertices of degree $3$ are adjacent, have weak local deficiency at most $1$.

\end{itemize}

	In Section 2, we prove our results for general graphs with small maximum degree
	and in Section 3 we consider complete multipartite graphs. Section 4 contains the proofs of our results on
	bipartite graphs.
	In Section 5 we consider graphs with large values of local and weak local deficiency
	and show that there are graphs with arbitrary large local deficiency, as well as weak local deficiency, 
	and that the difference
	between these two invariants may be arbitrarily large.
	In particular, we show that so-called generalized Hertz graphs have local deficiency 
	which grows linearly with the maximum degree, as well as with the number of vertices.
	We conclude the paper by pointing to some open questions and suggestions for further
	research.

\section{General graphs with small maximum degree}

	In this section we consider general graphs. 
First we give some preparatory definitions.
We use standard terminology of graph theory as in
e.g. \cite{West}.
$V(G)$ and $E(G)$ denote the sets of vertices and edges of a graph $G$, respectively. The degree of a vertex $v\in V(G)$ is denoted by $d_{G}(v)$ (or just $d(v)$), the maximum and minimum degrees of vertices in $G$ by $\Delta(G)$ and $\delta(G)$, respectively. A graph is {\em even} if all its degrees are even.

	A set $A$ of consecutive integers is called an {\em interval}, $A$ is a
	{\em near-interval}, if there is an integer $i$ such that $A \cup \{i\}$ is an interval,
	and it is a {\em weak near-interval} if there is a set $I=\{i_1,\dots, i_k\}$ of
	integers, no two of which are consecutive, 
	such that $A \cup I$ is an interval. A set $A$ of integers is a {\em cyclic interval
	modulo $t$} if it is an interval under the extra condition that color $t$ and $1$ are considered
	consecutive.

	Graphs with local deficiency $1$ thus have proper edge colorings where the colors
	of the edges incident with any vertex form a near-interval, a so-called {\em near-interval coloring};
	for graphs with weak local deficiency $1$, we call an edge coloring realizing this property
	a {\em weak near-interval coloring}.

	For an edge coloring $\varphi$ of a graph $G$ and any $v\in V(G)$, $\varphi(v)$ (or $\varphi_{G}(v))$ denotes the set of colors appearing on the edges incident with $v$. 
The smallest and largest colors of $\varphi(v)$
are denoted by $\underline \varphi(v)$ and $\overline \varphi(v)$, respectively.
    We say that $\varphi$ is {\em (weakly) near-interval}
	at $v$ if $\varphi(v)$ is a (weak) near-interval.

	We shall need the following theorem due to Fournier \cite{Fournier}.
\begin{theorem}
\label{th:fournier}
	If $G$ is a graph where no two vertices of maximum degree are adjacent, then $G$ is Class 1.
\end{theorem}

	It is known that Class 1 graphs with maximum degree at most $3$
	have interval colorings, which is sharp with respect to maximum degree
	\cite{AsratianCasselgrenPetrosyan}. In \cite{CasselgrenPetrosyan}, it was proved
	that graphs with maximum degree at most $4$ admit near-interval colorings.
	Moreover, the same holds for Class 1 graphs with maximum degree $5$ and no vertices of degree $3$
	\cite{CasselgrenPetrosyan}.
	
	We prove a corresponding result for weak near-interval colorings of graphs with maximum degree
	$5$. Relaxing the near-interval condition to weak near-interval implies that
	we do not need to restrict to Class 1 graphs, and we can allow some
	vertices of degree three.

\begin{theorem}
	Every graph with maximum degree at most $5$, where no two vertices of degree $3$ are
	adjacent has a weak near-interval coloring.
\end{theorem}
\begin{proof}
	Let $G$ be a graph as in the theorem;
	we color it as follows. Let $M$ be a maximum matching in the subgraph of $G$ induced
	by all vertices of degree $5$. Then $G-M$ satisfies that the vertices of degree $5$ induce
	an acyclic graph, so it has a proper $5$-edge coloring. We color the edges of $M$ by color $6$.
	This yields a proper $6$-edge coloring $\varphi$ of $G$, where only vertices of degree $5$ are incident with
	edges of color $6$.

	A vertex $v$ of $G$ where $\varphi$ is not weakly near-interval, we call {\em $\varphi$-problematic} (or just {\em problematic}).
	
	Straightforward case analysis yields that if $\varphi$ is problematic at a vertex $v$, then
\begin{itemize}

	\item $\varphi(v)=\{1,4\}$, $f(v)=\{1,5\}$, $f(v)=\{2,5\}$, if $v$ has degree $2$

	\item $\varphi(v)=\{1,4,5\}$, $f(v)=\{1,2,5\}$, if $v$ has degree $3$.

\end{itemize}

	An edge of color $1$ or $5$ that is incident with a problematic vertex is called {\em bad}.
	
	Using the coloring $\varphi$ as starting point, we shall prove that $G$ has a weak near-interval
	coloring by steps. We shall use colors $-1, \dots, 7$ in the process of constructing such a coloring.

	We first prove the following claim.

\begin{claim}
\label{cl:prob2}
	There is a proper edge coloring $g$ of $G$ with colors $0, \dots, 6$ such that
\begin{itemize}
	\item[(a)] no vertices of degree $2$ are problematic

	\item[(b)] colors $0$ and $6$ only appear on edges $e$ where one endpoint has degree at least $4$, and $g$
	is weakly near-interval at both endpoints of $e$,

	\item[(c)] no vertex is incident with exactly four edges that are colored $0,1,2,3$ or $3,4,5,6$.
\end{itemize}
\end{claim}


\begin{proof}	
	Let us assume that $g$ is a proper edge coloring, obtained from $\varphi$ 
	by recoloring some bad edges incident with vertices of degree $2$,
	with the fewest problematic vertices of degree $2$.
	Let us first note that if two problematic vertices are adjacent, then we can easily obtain a proper coloring
	with fewer problematic vertices, a contradiction. Thus we assume that this is not the case.

	Assuming that there is at least one problematic vertex of degree $2$ under $g$,
	we shall argue that there is a coloring $g'$ using colors $0, \dots, 6$ satisfying (a)-(c) and with
	fewer problematic vertices of degree $2$.

\bigskip

	{\bf Case 1.} {\em There is a problematic vertex $v$ which is incident with edges colored $1$ and $4$:}

	Suppose that $uv$ is colored $1$ under $f$, and that $vw$ is colored $4$. If color $2$ or $3$ does
	not appear at $u$, then we can recolor $uv$ by color $2$ or $3$ to obtain a coloring with fewer
	problematic vertices, so we assume that this is the case; analogously for $w$. Similary, color $5$
	must appear at $u$. If color $6$ does not appear at $u$, then we recolor $uv$ by color $6$.

	Otherwise, assume that color $6$ appears at $u$, so $g(u)=\{1,2,3,5,6\}$. If $5$ does not appear
	at $w$, then we recolor $uv$ by $4$ and $vw$ by $5$, so we may assume that this holds. Now, if
	$6$ does not appear at $w$, then we recolor $uv$ by $4$ and $vw$ by $6$, so we assume that
	$g(w) = \{2,3,4,5,6\}$. Hence, we can recolor $vw$ by the color $0$ to obtain a coloring
	with fewer problematic vertices, and where $v$ is not problematic.

	\bigskip

	{\bf Case 2.}	 {\em There is a problematic vertex $v$ which is incident with edges colored $2$ and $5$:}
	
	This case is completely analogous to the preceding case, where every color $c$ is replaced by $6-c$.

	\bigskip

	{\bf Case 3.}	 {\em There is a problematic vertex $v$ which is incident with edges colored $1$ and $5$:}
	
	Suppose $uv$ is colored $1$, and $vw$ is colored $5$. As in the preceding cases, we may assume that
	colors $3,4$ both appear at $u$, and colors $2,3$ both appear at $w$. Furthermore, we may
	assume that $4$ appears at $w$, and $2$ appears at $u$, because otherwise we may recolor $uv$
	or $vw$ and proceed as in Case 1 or Case 2. If $0$ does not appear at $w$, then we may recolor $vw$
	by this color, so we may assume that $0$ appears at $w$; similarly with color $6$ at $u$.
	Thus we may recolor $uv$ by $0$, and $vw$ by color $1$, to obtain a coloring with fewer
	problematic vertices.

	\bigskip
	
	This shows that there is a coloring $g'$ as required, which contradicts the choice of $g$. 
\end{proof}

	From the preceding claim we infer that
	there is a proper edge coloring $g$ of $G$ using colors $0,\dots, 6$ such that
\begin{itemize}

	\item $g(v)=\{1,4,5\}$ or $g(v)=\{1,2,5\}$, if $v$ is problematic under $g$,

	\item colors $0$ and $6$ only appear on edges $e$ where one endpoint has degree at least $4$, and $g$
	is weakly near-interval at both endpoints of $e$,

	\item no vertex is incident with exactly four edges that are colored $0,1,2,3$ or $3,4,5,6$.

\end{itemize}

	Next, we prove the following claim.

\begin{claim}
\label{cl:prob3}
	There is a proper edge coloring $f$ of $G$ with colors $-1, \dots, 7$ such that
\begin{itemize}
	\item[(i)] no vertices of $G$ are problematic

	\item[(ii)] colors $0$ and $6$ only appear on edges $e$ with one endpoint of degree at least $4$, 
	and $f$ is weakly near-interval at both endpoints of $e$, 

	\item[(iii)] no vertex is incident with exactly four edges that are colored $0,1,2,3$ or $3,4,5,6$.

	\item[(iv)] colors $-1$ and $7$ only  appear on edges that are not adjacent to bad edges.
\end{itemize}
\end{claim}
\begin{proof}
	Among the colorings satisfying (ii)-(iv), and with no problematic vertices of degree
	$2,4$ or $5$,
	we choose one $f$ with the minimum number
	of problematic vertices of degree $3$. Assuming that at least one vertex is problematic under $f$,
	we shall prove that there is a coloring $f'$ with fewer problematic vertices of degree $3$ 
	satisfying (ii)-(iv) (and with no problematic vertices of degree $2$, $4$, or $5$), 
	thereby deriving a contradiction.

	
	\bigskip

	{\bf Case 1.}	 {\em There is a problematic vertex $v$ with $f(v) =\{1,2, 5\}$:}

	Suppose that $uv$ is colored $1$, $vw$ is colored $2$, and $vx$ is colored $5$. As before,
	we may assume that $u,w$ and $x$ are all incident with an edge colored $3$. Moreover, color $4$
	appears at both $u$ and $x$.

	\bigskip

	{\bf Case 1.1.} {\em Color $1$ or $2$ appears at $x$:}

	Suppose that color $1$ or $2$ appears at $x$. If color $0$ does not appear at $x$, then we recolor
	$vx$ by $0$ and are done. So suppose that $0$ appears at $x$. If color $2$ appears at $x$, then
	we recolor $vx$ by $-1$. Otherwise, if $1$ appears at $x$, then we recolor $xv$ by $-1$ unless
	$x$ is adjacent to a problematic vertex $y$ via the edge colored $1$.
	
	So assume that there is such a problematic vertex $y$.
	If $f(y) = \{1,4,5\}$, then we recolor $xy$ by $2$ and $vx$ by $-1$, and are done.
	Otherwise if $f(y) = \{1,2,5\}$, then we set $v_1 = v$, $x_1 = x$, and consider a maximal
	$(1,5)$-colored path $P = v_1 x_1 v_2 x_2 \dots v_k x_k$ satisfying that
\begin{itemize}

	\item $v_1, \dots v_k,$ are problematic

	\item colors $f(v_i)=\{1,2,5\}$, $i=1,\dots, k-1$

	\item colors $f(v_k)=\{1,2,5\}$ or $f(v_k)=\{1,4,5\}$.

\end{itemize}	
	Note that we may assume that $\{3,4\} \subseteq f(x_i)$,  $i=1,\dots k$, 
	since otherwise we can recolor $v_i x_i$
	and obtain a coloring with fewer problematic vertices.

	Now we recolor the edges $v_1x_1, \dots v_k x_k$ as follows:

\begin{itemize}

	\item If $0$ appears at $x_i$, then we recolor $v_i x_i$ by color $-1$, otherwise 
	we recolor it by $0$, for $i=1,\dots, k-1$.
	
	\item If $1,4,5$ appear at $v_k$ and $2$ does not appear at $x_{k-1}$, then we recolor $x_{k-1}v_k$
 	by $2$. On the other
	hand, if $2$ appears at $x_{k-1}$, then we recolor it $6$.

	\item If $1,2,5$ appear at $v_k$, then since $x_5$ has degree at least $4$, and 
	$f(x_k) \neq \{3,4,5,6\}$, $1$ or $2$ must appear at $x_5$.
	If $0$ does not appear at $x_k$, then we recolor $v_kx_k$ by $0$;
	otherwise if $0$ does appear at $x_k$, then we recolor $x_k v_k$ by $-1$.
\end{itemize}

	It is readily verified that this yields a required coloring $f'$, which contradicts the choice of $f$.

	\bigskip

	{\bf Case 1.2.} {\em Neither of color $1$ and $2$ appear at $x$:}

	Note that the conditions (i)-(iv) imply that $0$ does not appear at $x$. 
	Moreover, since $x$ does not have degree $3$, $6$ must appear at $x$.
	By the same conditions, $7$ cannot appear at $x$, so
	$x$ has degree $4$ and colors $3,4,5,6$ are present at $x$.
	However, this contradicts the condition (iii) above.
	We conclude that Case 1.2 is not possible.

	\bigskip

	{\bf Case 2.}	 {\em There is a problematic vertex $v$ with $f(v) = \{1,4,5\}$:}

	This case is completely analogous to Case 1 by replacing every color $c$ in that case by $6-c$.
\end{proof}	
\bigskip

	By the preceding claim we may conclude that there is a coloring satisfying (i)-(iv) with no problematic
	vertices of degree $3$. This is the required weak near-interval coloring.
\end{proof}

	Allowing adjacent vertices of degree $3$ in the above proof would introduce significant
	difficulties. One difficult configuration is when three edges colored $1,2,5$ are incident
	to a common vertex, and all three other endpoints of these edges have degree $3$.
	If all these endpoints are incident with edges colored $3$ and $4$, we would have to
	recolor a quite large subgraph and consider a large number of different cases.

Finally, let us note that Vizing's edge coloring theorem implies that all regular graphs have local deficiency at most $1$.
More generally, it seems that graphs $G$ whose vertex degrees are not too far apart have small local deficiency; indeed, we have the following.

\begin{proposition}
\label{th:D-d<=1}
If $G$ is a graph with $\Delta(G)-\delta(G)\leq 1$, then $G$ has  local deficiency at most $1$.	
\end{proposition}
This theorem can be proved by removing a maximum matching in the subgraph of $G$ induced by the vertices of maximum degree $\Delta(G)$,
and coloring its edges by color $\Delta(G)+1$.
The remaining graphs is $\Delta(G)$-edge-colorable (by Vizing's theorem or by Theorem \ref{th:fournier}). Thus $G$ has a near-interval coloring.


%


\section{Complete multipartite graphs}

	In this section we prove that complete multipartite graphs have
 	small weak local deficiency. Complete bipartite graphs are well-known
	to have interval colorings, while there are complete tripartite
	graphs that do not admit such colorings (since they are Class 2).
	We prove the following.

\begin{theorem}
\label{prop:3partite}
	Every complete $3$-partite graph has weak local deficiency $1$.
\end{theorem}
\begin{proof}
	Let $G=K_{a,b,c}$ be a tripartite graph where with parts $A,B,C$ of sizes $a,b,c$, respectively, 
	where $a \geq b \geq c$. Let $A=\{u_1, \dots, u_a\}$ and $B=\{v_1,\dots, v_b\}$.
	Define an interval edge coloring of the copy of $K_{a,b}$ contained in $G$
	by setting $\varphi(u_iv_j) = i+j-1$. 
	Next, multiply each color by $2$, so
	the obtained coloring $\varphi'$ is a weak near-interval coloring using colors $\{2,\dots, 2a+2b-2\}$.

	Next, we color the edges of the subgraph $H$ in $G$ that is isomorphic to $K_{a+b,c}$. 
	Note that the smallest color that appears at $u_i$ under $\varphi'$ is $2i$, $i=1,\dots,a$,
	and similarly, the smallest color appearing at $v_i$ is $2i$.
	Let $C= \{w_1,\dots, w_c\}$, 
	and let $X=\{x_1, \dots x_{a+b}\}$,
	where $x_i = u_i$, $i=1\dots, a$ and $x_{a+i} = v_i$, $i=1\dots, b$.
	
	Let $\gamma$ be an interval edge coloring of the complete bipartite graph $K$ with 
	parts $X$ and $C$ contained in $G$
	defined by $\gamma(x_i w_j) = i+j-1$,
	and set $\gamma'(e) = 2 \gamma(e)-1$ for every edge $e \in E(K)$.

	Let us verify that $\varphi'$ and $\gamma'$ together form an weak near-interval coloring of $G$.
	It is clear that both $\varphi'$ and $\gamma'$ are weakly near-interval at every vertex of $G$, so
	the coloring is weakly near-interval at vertices in $C$.
	Moreover, it is straightforward that $\varphi'$ and $\gamma'$ is weakly near-interval at vertices in $A$.

	Let us now consider the vertices in $B$.
	The largest color appearing at $v_i$ under $\varphi'$ is $2a+2i-2$, and the smallest color
	appearing at $v_i$ under $\gamma'$ is $2(a+1) + 2i-2 -1= 2a +2i-1$. 
	Consequently, the coloring is weakly near-interval at
	the vertices in $B$.
\end{proof}

More generally we have the following.

\begin{theorem}
	If $G$ is a complete multipartite graph with a unique part of largest size,
	then $G$ has weak local deficiency $1$.
\end{theorem}
\begin{proof} 
	Suppose that $X$ is a part of largest size in $G$ and consider the complete multipartite graph
	$H=G-X$. Label the vertices of $H$ by $\{v_1,\dots, v_n\}$ and the vertices of $X$ by
	$\{x_1,\dots, x_t\}$.
	We color the edges of $H$ by setting $\alpha(v_i v_j) = i+j-1 \pmod n$. Then $\alpha(v)$
	is a cyclic interval modulo $n$ for each $v \in V(H)$.
	Moreover, if we consider colors as
	consecutive modulo $n$, then every color $i$, $1\leq i \leq n$, is the ``first'' color of such a
	cyclic interval modulo $n$ appearing at some vertex of $H$ (in the sense that such a ``first
	color'' does not have a consecutive predecessor). 
	Without loss of generality, we assume that the colors used by $\alpha$
	are permuted cyclically modulo $n$ so that $i$ is the ``first'' such color appearing at $v_i$. Note that
	since $X$ is a largest part in $G$, strictly less than $t$ colors from
	$\{1,\dots, n\}$ are missing at every vertex of $H$ under $\alpha$.
	
	Next, we color the complete bipartite graph with parts $X$ and $V(H)$ by setting
	$\gamma(v_ix_j) = i-j \pmod n$. 
	By construction the colorings $\gamma$ and $\alpha$ together
	form an improper interval edge coloring of $G$. Next, we define
	two new proper edge colorings $\gamma'(e) = 2\gamma(e)$ and
	$\alpha'(e) = 2\alpha(e)-1$. Since $\gamma$ and $\alpha$ together form an improper interval edge coloring, 
	and strictly less than $t$ colors from
	$\{1,\dots, n\}$ are missing at every vertex of $H$ under $\alpha$,
	the colorings $\alpha'$ and $\gamma'$ together form a weak near-interval coloring.
\end{proof}

The following upper bound applies to all complete multipartite graphs.

\begin{theorem}
	Every complete multipartite graphs has weak local deficiency at most $2$.
\end{theorem}

	This theorem can be proved as the preceding theorem. 
	Using the same notatation as in that proof, let $X$ be a part
	of maximum size in a complete multipartite graph $G$ and 
	consider $H=G-X$. By defining the colorings $\alpha$, $\gamma$,
	and thereafter $\alpha'$ and $\gamma'$,
	as in the proof of the preceding theorem,
	we obtain a proper edge coloring of $G$. Since $X$ is a largest part in $G$,
	it can be verified that the resulting edge coloring has weak local deficiency at most $2$.
	We omit the details.


\section{Bipartite graphs}

	In \cite{CasselgrenPetrosyan}, it was proved
	that all bipartite graphs with maximum degree $5$, as well as Eulerian bipartite graphs
	with maximum degree $6$ are near-interval colorable. Here, we shall 
	prove stronger results in the context of weak local deficiency.
	First we note the following general upper bound, which can be proved using
	Petersen's well-known $2$-factor theorem, 
	which also holds in the setting of multigraphs
	with multiple edges and loops.

\begin{proposition}
\label{th:Eul}
If $G$ is an Eulerian bipartite graph, then $\mathrm{def}_{wloc}(G)\leq \frac{\Delta(G)-\delta(G)}{2}$.	
\end{proposition}
\begin{proof}
For the proof, we construct a new multigraph $G^{\star}$ as follows: 
for each vertex $u \in V(G)$ of degree $2k$, we add $\frac{\Delta(G)}{2}-k$ loops at $u$ 
$\left(1\leq k< \frac{\Delta(G)}{2}\right)$. Since $G^{\star}$ is $\Delta(G)$-regular, 
Petersen's theorem yields that it can be decomposed into
a union of edge-disjoint $2$-factors $F_{1},\ldots,F_{\frac{\Delta(G)}{2}}$. By removing all
loops from the $2$-factors $F_{1},\ldots,F_{\frac{\Delta(G)}{2}}$, we
deduce that the resulting graph $G$ is a union of edge-disjoint
subgraphs $F^{\prime}_{1},\ldots,F^{\prime}_{\frac{\Delta(G)}{2}}$, where each
$F'_{i}$ is a collection of even cycles in $G$.
For each $i$ 
we color the edges of $F^{\prime}_{i}$ alternately with colors
$i$ and $\frac{\Delta(G)}{2}+i$; let $\alpha$ be the resulting proper edge coloring of $G$. 
It is straightforward that $\alpha$ is a coloring realizing the required upper bound on
$\mathrm{def}_{wloc}$.
\end{proof}

\begin{corollary}
\label{prop:bip1}
Every bipartite graph with even maximum degree $\Delta(G)$, where $\Delta(G) - \delta(G)\leq 2$,
has weak local deficiency at most $1$.
\end{corollary}
\begin{proof}
Define an auxiliary graph $G'$ as follows:
take two isomorphic copies $G_{1}$ and $G_{2}$ of the graph 
$G$ and join by an edge every vertex with degree $\Delta(G)-1$ in $G_{1}$ with its copy in $G_{2}$.
The resulting graph is Eulerian, so by Proposition \ref{th:Eul},
it has a weak near-interval coloring $\alpha$ with colors $1,\ldots,\Delta(G)$. 
The restriction of this coloring to $G_1$ is a weak near-interval coloring, so
there is a weak near-interval coloring of $G$ with colors $1,\ldots,\Delta(G)$. 
\end{proof}

Corollary \ref{prop:bip1} can be slightly generalized as follows.

\begin{proposition}
\label{<=2&>=D-2} 
\begin{itemize}
    \item[(a)] If $G$ is a bipartite graph with
$\Delta(G)=2r$ ($r\geq 2$) and for every $v\in V(G)$, $d_{G}(v)\in \{1,2,2r-2,2r-1,2r\}$, then $\mathrm{def}_{wloc}(G)\leq 1$.
    \item[(b)] If $G$ is a bipartite graph with
$\Delta(G)=2r-1$ ($r\geq 2$) and for every $v\in V(G)$, $d_{G}(v)\in
\{1,2,2r-2,2r-1\}$, then $\mathrm{def}_{wloc}(G)\leq 1$.
\end{itemize}
\end{proposition}

\begin{proof} 
First we prove (a). For that, define an auxiliary  multigraph $G^{\star}$ as follows:

\begin{itemize}

	\item take two isomorphic copies $G_{1}$ and $G_{2}$ of the graph $G$ and join by an edge every vertex with an odd vertex degree in
$G_{1}$ with its copy in $G_{2}$,

\item for each vertex $u \in
V(G_{1}) \cup V(G_{2})$ of degree $2$, we add $r-1$ loops at $u$,
and 

\item for each vertex $v\in V(G_{1}) \cup V(G_{2})$ of degree $2r-2$, we add one loop at $v$. 

\end{itemize}

Clearly, $G^{\star}$ is a $2r$-regular
multigraph, so by Petersen's $2$-factor theorem, $G^{\star}$ can be represented as a union of edge-disjoint $2$-factors $F_{1},\ldots,F_{r}$. By removing all loops from $2$-factors $F_{1},\ldots,F_{r}$ of $G^{\star}$, we obtain that the resulting graph $G^{\prime}$ is a union of edge-disjoint even subgraphs $F^{\prime}_{1},\ldots,F^{\prime}_{r}$. Since $G^{\prime}$ is bipartite, for each $i$ ($1\leq i\leq r$),
$F^{\prime}_{i}$ is a collection of even cycles in $G^{\prime}$.


We define an edge coloring $\alpha$ of $G'$ by coloring the edges of each $F'_{2i-1}$ properly by color $4i-3$ and $4i-1$,
and each $F'_{2i}$ by colors $4i-2$ and $4i$, except that we color $F'_r$ by colors $2r-1$ and $2r$ if $r$ is odd.
Then $\alpha$ is a proper edge coloring of $G^{\prime}$ with colors $1,\ldots,2r$, and for each vertex $v\in V(G^{\prime})$ with $d_{G^{\prime}}(v)=2r$, $\alpha_{G^{\prime}}(v)=[1,2r]$.

Consider the restriction of $\alpha$ to $G_1$. By construction, a vertex of degree $2$ is assigned two consecutive colors or two consecutive even or odd integers.
Similarly, a vertex of degree $2r-2$ is missing two consecutive integers, or two consecutives odd or even numbers. In conclusion, $\alpha$ is
a weak near-interval coloring.

Next, we prove (b).
Let us construct an auxiliary graph $H$ with maximum degree $2r$ as follows: we take two
isomorphic copies of the graph $G$ and join by an edge one vertex of
degree $2r-1$ with its copy. 
By part (a) of Theorem \ref{<=2&>=D-2}, $H$ has a weak near-interval $2r$-coloring. 
The
restriction of this proper edge coloring to the edges of the
graph $G$ is a weak near-interval coloring of $G$.
\end{proof}

Next, we have the following.

\begin{theorem}
\label{th:bip6}
	Every bipartite graph with maximum degree at most $6$ has weak local deficiency at most $1$.
\end{theorem}
\begin{proof}
	Let $G$ be a bipartite graph with maximum degree at most $6$, and 
	let $M_1$ be a matching covering all vertices of degree $6$,
	and such that every edge in $M_1$ has an endpoint of degree $6$. Similarly, let 
	$M_2$ be a matching in $G-M_1$ covering all vertices
	of degree $5$, where every edge in $M_2$ has at least one endpoint of degree $5$. 
	The existence of such matchings follows from 
	e.g.~K\"onig's edge coloring theorem for bipartite graphs.
	Set $F_1 = G[M_1 \cup M_2]$. 
	
	Then $H=G - E(F_1)$ has maximum degree $4$, and by proceeding as
	in the proof of Proposition \ref{th:Eul}, using
	Petersen's $2$-factor theorem, we can decompose
	$G$ into two bipartite subgraphs $F_2$ and $F_3$ 
	with maximum degree $2$, and where every vertex of degree $2$ in $H$ is in exactly
	one of $F_2$ and $F_3$.
	
We color $F_2$ properly by colors $3,4$, and $F_3$ properly by colors $5,6$, so that the following holds:

\begin{itemize}
	\item[(*)]	 if $x, y \in V(G)$ satisfy that $d_G(x) = d_G(y) =2$,
	$x$ and $y$ are endpoints of a path in $F_1$ of even length, and $x$ and $y$ are also endpoints of (possibly distinct) paths in $F_2$ ($F_3$), then
	the edges of $F_2$ ($F_3$) incident with $x$ and $y$, respectively, are colored differently. 
\end{itemize}
	This is clearly possible, since we may color paths  in $F_2$ and $F_3$ 
	sequentially using colors $2,3$ and $5,6$, respectively.

\bigskip

	Next, we color the components of $F_1$. Note that by construction, every edge in $F_1$ has at least one 
	endpoint of degree at least $5$. 
	
	Suppose that $C$ is a cycle in $F_1$. Then every edge in $F_1$ 
	has an endpoint of degree $6$. Fix a cyclic orientation of $C$ and
	assume further that $C$ has a vertex $x$ of degree at most $4$ in $G$ where color $3$ or $4$ (or both)
	appears; then
	we color the edges incident with $x$ by $1$ and $2$. Next, we continue coloring the edges $1,2$ alternately,
	according to the cyclic orientation,
	until we reach a vertex $y$ which in $C$ is adjacent to a vertex $z$ of degree at most $4$ in $G$ that is
	incident with an edge colored $5$ or $6$.
	(Note that this means that $y$ has degree $6$ in $G$.)
	Then we instead color $yz$ by $7$, and continue coloring the edges alternately by $7$ and $8$, until we
	reach a vertex in $C$ that is
	adjacent (in $C$) to a vertex $w$ of degree at most $4$ in $G$ that is incident with an edge 
	colored $3$ or $4$, in which
	case we switch back to coloring by colors $1,2$ alternately. 

	We continue in this manner until $C$ is properly colored. If the last uncolored portion of 
	$C$ is colored by $1$ and $2$, then we choose the coloring
	of these edges so that the resulting coloring is proper.

	If instead $C$ has no vertex of degree at most $4$ where color $3$ or $4$ appears, 
	then we color the edges of $C$ by $7$ and $8$
	alternately.
	
	In both cases, this yield a weak near-interval coloring of $G[V(C)]$, because every vertex in $C$ 
	either has degree 
	$2,3,4$ or $6$. If a vertex has degree smaller than $6$ and greater than $2$,
	then its incident edges from $G-E(F_1)$ are either colored by colors from $\{3,4\}$ or 
	$\{5,6\}$, but not from both sets. In the first case, incident edges from $C$ are colored $1,2$,
	and in the second case by $7,8$.
	
	If a vertex in $C$ has degree $6$ in $G$, then its incident
	edges in $G-E(F_1)$ are colored $3,4,5$ and $6$, and in $C$ it is incident with 
	two edges with colors from $\{1,2,7,8\}$, which
	in all cases yield a weak near-interval.

	\bigskip

	Let us now consider a path $P$ in $F_1$. If there are no vertices where only colors 
	from $\{3,4\}$ ($\{5,6\}$) appear,
	then we color the edges of $P$ alternately by $7,8$ ($1,2$). Moreover, we color $P$ 
	so that if one endpoint is only adjacent to an edge colored $5$ ($6$) then we 
	color the adjacent edge of $P$ by color $7$ ($8$); and
	similarly for colors $1,2$ if there are no vertices where only colors from $\{5,6\}$ appear. 
	If $P$ has even length, then
	this yields a weak near-interval coloring of $G[V(P)]$,
	since the condition (*) holds.
	If $P$ instead has odd length, then two vertices of degree $6$ are adjacent on $P$, and we color
	the edge between these vertices by $7$, and all other edges by $1,2$, so that the last
	edge of $P$ is colored $2$. This yields a weak near-interval coloring of  $G[V(P)]$.

	Let us now assume that there are vertices in $P$ where only colors from $\{3,4\}$ appear, 
	and also where only colors from $\{5,6\}$ appear.
	We proceed similarly to the case of cycles. Suppose e.g. that color $3$ ($4$) appear at 
	one endpoint of $P$; then we color the first edge
	by $1$ ($2$), and then continue coloring the edges along 
	$P$ by colors $1,2$ alternately, until we reach a vertex
	$y$ which in $P$ is adjacent to a vertex $z$ of degree at most $4$ in $G$ that is incident with 
	edges only colored by colors from $\{5,6\}$.
	Then we switch to coloring the edges alternately by colors $7,8$ as in the case of cycles. 
	We continue coloring in this manner until
	$P$ is properly colored, choosing the $2$-coloring for the last portion of $P$ so that 
	the last edge of $P$ is either colored $2$
	or $7$.
	As in the case of cycles, this yields a weak near-interval coloring of $G[V(P)]$.

\bigskip
	
	We conclude that we may color all cycles and paths of $F_1$ according to the description above to
	obtain a weak near-interval coloring of $G$.	
\end{proof}

	Next, we consider Eulerian bipartite graphs. 
	We shall use the following proposition, which follows from a decomposition technique
	in \cite{AsratianCasselgrenPetrosyan2} (see Theorem 2.7 in \cite{AsratianCasselgrenPetrosyan2}).

\begin{proposition}
\label{prop:decomp}
	Every Eulerian bipartite graph with maximum degree $8$ has a decomposition
	into two subgraphs where all vertices have degree $2$ or $4$.
\end{proposition}

We shall also use the following lemma from \cite{CasselgrenPetrosyan}.

\begin{lemma}
\label{lem:4}
	If $G$ is a bipartite graph with $\Delta(G) \leq 4$, then $G$ 
	has a proper 4-edge-coloring such that every vertex of degree $2$
	is incident with edges colored $1,2$, or colored $3,4$.
\end{lemma}

\begin{theorem}
\label{th:WeakEul}
	Every Eulerian bipartite graph with maximum degree at most $8$ has weak local deficiency $1$.
\end{theorem}

\begin{proof}
	By the preceding theorem, we may assume $G$ has maximum degree $8$.
	Using Proposition \ref{prop:decomp}, we
	decompose $G$ into a blue subgraph $B$ and a red subgraph $R$ so that 
	both $R$ and $B$ satisfy that all vertices have degree $2$ and $4$ in these subgraphs.
	Using Lemma \ref{lem:4}, we color $R$ properly by colors $1,3, 6,8$,
	so that all vertices with degree two are incident with edges colored $1,3$ or $6,8$; 
 	similarly, we color $B$ using the colors $2,4, 5,7$ with color pairs $2,4$ and $5,7$.

	The obtained edge coloring weakly near-interval.
\end{proof}

\begin{corollary}
	Every  bipartite graph with maximum degree at most $8$, with no vertices of degree $5$ and
	where no two vertices of degree $3$ are adjacent, has weak local deficiency $1$.
\end{corollary}
\begin{proof}
	Take two copies $G_1$ and $G_2$ of a graph $G$ as described in the corollary, 
	and join corresponding odd-degree 
	vertices by an edge, so that the resulting graph $H$ is Eulerian. We color $H$
	as in the proof of the preceding
	theorem, and take the restriction of this coloring to a copy of $G$. It is readily verified that this coloring,
	denoted $f$, is weakly near-interval at every vertex that does not have degree $3$.
	In fact, if $f(v)$ is not a weak near-interval, then
	$$f(v) \in \left\{ \{1,3,7\}, \{1,5,7\}, \{2,4,8\}, \{2,6,8\}  \right\}.$$

	We shall recolor edges colored $1,2,7,8$ incident to such vertices to obtain a weak near-interval
	coloring.  Since no two vertices of degree $3$ are adjacent, every recolored
	edge has at least one endpoint with degree distinct from $3$.

	We start by recoloring the edges colored $8$ incident with vertices $v$ with $f(v)=\{2,4,8\}$.
	Assume $uv$ is such an edge colored $8$. If one of the colors $5,3,1$ is missing at $u$,
	then we recolor $uv$ by this color (in that order of preference, since a vertex of degree two,
	where color $8$ appears, is incident with an edge of color $6$); otherwise recolor $uv$ by $0$. 
	Since $u$ does not have degree $3$, it follows that the resulting coloring is weakly near-interval
	at $u$.

	Let us now consider the edges colored $7$ that are incident with vertices $v$ where $1,3,7$ appears.
	Assume $uv$ is such an edge colored $7$. If one of the colors $4,2,5$ is missing at $u$,
	then we recolor $uv$ by this color (in that order of preference, since a vertex of degree two,
	where color $7$ appears, is incident with an edge colored $5$). Otherwise,	
	if color $0$ does not appear at $u$, we recolor $uv$ by $0$, and if $0$ appears at $u$,
	we recolor $uv$ by the color $-1$. The resulting coloring is weakly near-interval at $u$.

	We note the following two properties of the hitherto constructed coloring:
\begin{itemize}
	\item[(a)]	If $0$ appears at a vertex $u$, then either colors $1,3,5$ or colors $2,4,5$ appear at $u$
	as well.

	\item[(b)]	If color $-1$ appears at a vertex $u$, then colors $0,2,4,5$ appear at $u$ as well.
\end{itemize}

	Now we consider the edges colored $1$ that are incident with vertices $v$ where $1,5,7$ appear.
	Assume $uv$ is such an edge colored $1$. If one of the colors $4,6,8$ is missing at $u$,
	then we recolor $uv$ by this color (in that order of preference). Otherwise,	
	we recolor $uv$ by the color $9$. The resulting coloring is weakly near-interval at $u$.

	Finally, let us consider the edges colored $2$ that are incident with vertices $v$ where $2,6,8$ appears.
	Assume $uv$ is such an edge colored $2$. If one of the colors $5,4,7$ is missing at $u$,
	then we recolor $uv$ by this color (in that order of preference).
	Otherwise, if $9$ does not appear at $u$, then we recolor $uv$ by $9$, otherwise
	we recolor by $10$.

	Note that if $9$ appears at a vertex, then $4,5,7$ appears at this vertex, and if $0$ in addition
	appears at the same vertex, then in addition $1$ and $3$, or $2$, appears at $u$.
	Furthermore, $-1$ only appears at a vertex where $0$ is present, and $10$ only appears at a vertex
	where $9$ is present.

	In conclusion, the obtained coloring is weakly near-interval.
\end{proof}


\section{Graphs with large local and weak local deficiency}

Besides trivial examples of graphs without interval colorings such as $K_3$, there are several well-known infinite families of bipartite graphs without interval colorings
(see e.g. \cite{GiaroKubaleMalaf1, PetrosHrant}). Here we study the local and weak local deficiency
of some such families.
	
	We begin by considering two families of graphs which are known to have
	large deficiency \cite{GiaroKubaleMalaf1}.
For any $a,b,c\in \mathbb{N}$, define the graph $S_{a,b,c}$ as follows:

$$V(S_{a,b,c})=\{u_{0},u_{1},u_{2},u_{3},v_{1},v_{2},v_{3}\}\cup \{x_{1},\ldots, x_{a},y_{1},\ldots, y_{b},z_{1},\ldots,z_{c}\}
$$
and
\begin{align*}
E(S_{a,b,c}) = & \{u_{1}v_{1},v_{1}u_{2},u_{2}v_{2},v_{2}u_{3},u_{3}v_{3},v_{3}u_{1}\}
\cup \{u_{0}x_{i},u_{1}x_{i}:1\leq i\leq a\}
\\
& \cup\{u_{0}y_{j},u_{2}y_{j}:1\leq j\leq b\}\cup\{u_{0}z_{k},u_{3}z_{k}:1\leq k\leq c\}.
\end{align*}
The graphs $S_{a,b,c}$ are known as the Sevastjanov rosettes. Figure \ref{fig1} shows the graph $S_{7,7,7}$.

\begin{figure}[h]
\begin{center}
\includegraphics[width=27pc]{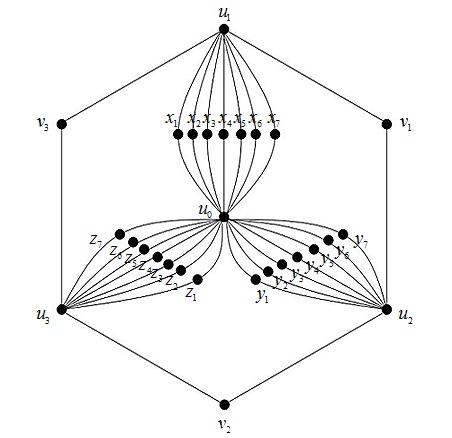}\\
\caption{The graph $S_{7,7,7}$.}\label{fig1}
\end{center}
\end{figure}

Next we define a family of graphs $\Delta_{a,b,c}$  ($a,b,c\in \mathbb{N}$) as follows.
We set
$$V(\Delta_{a,b,c})=\{u_{0},u_{1},u_{2},u_{3}\}\cup \{x_{1},\ldots, x_{a},y_{1},\ldots, y_{b},z_{1},\ldots,z_{c}\}$$ 
and 
\begin{align*}
E(\Delta_{a,b,c}) = & \{u_{0}x_{i},u_{1}x_{i},u_{2}x_{i}:1\leq i\leq a\}\cup\{u_{0}y_{j},u_{2}y_{j},u_{3}y_{j}:1\leq j\leq b\}
\\
 & \cup\{u_{0}z_{k},u_{3}z_{k},u_{1}z_{k}:1\leq k\leq c\}.
\end{align*}
The graphs $\Delta_{a,b,c}$ are known as the Malafiejski's rosettes.
Figure \ref{fig2} shows the graph $\Delta_{5,5,5}$.

\begin{figure}[h]
\begin{center}
\includegraphics[width=27pc]{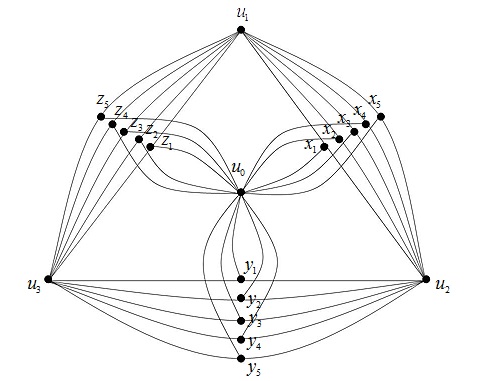}\\
\caption{The graph $\Delta_{5,5,5}$.}\label{fig2}
\end{center}
\end{figure}

Giaro et al.~\cite{GiaroKubaleMalaf1} showed that 
the graphs $S_k=S_{k,k,k}$ and $\Delta_k=\Delta_{k,k,k}$ 
satisfy
$\mathrm{def}(S_{k})\geq k-6$ and $\mathrm{def}(\Delta_{k})\geq k-4$ 
for each $k\geq 6$. Here, we show that these graphs have large local deficiency as well.
For an edge coloring $\varphi$ of a graph, denote by $\overline{\varphi}(v)$ and $\underline{\varphi}(v)$,
the largest and the smallest colors appearing on edges incident with $v$, respectively.

\begin{proposition}
\label{prop:rosettes} Let $k$ be a nonnegative integer.
\begin{itemize}
\item[(i)] If $a\geq 5k+7$, then $\mathrm{def}_{wloc}(S_{a,b,c})\geq k+1$.

\item[(ii)] If $r \geq 3k+5$, then $\mathrm{def}_{wloc}(\Delta_{r,s,t})\geq k+1$.
	\end{itemize}

\end{proposition}
This proposition thus implies that there are graphs $G$ with local deficiency  $\Omega(\Delta(G)/9)$.
\begin{proof} 
We first prove (i).
Let $G=S_{a,b,c}$. Suppose, to the contrary, that the graph $G$ has local deficiency at most $k$,
and consider a $q$-coloring $\alpha$ of $G$ realizing this property. Then $q\geq a+b+c$.

Let us consider two vertices $y_{i_0}$ and $z_{j_0}$ 
such that
$\alpha(y_{i_0}u_0)=\underline{\alpha}(u_0)$ and $\alpha(z_{j_0}u_0)=\overline{\alpha}(u_0)$. Let $p=\underline{\alpha}(u_0)$.
By the definition of $\alpha$, we have
$$p+a+b+c-1\leq \overline{\alpha}(u_0) \leq p+a+b+c+k-1.$$
Consider the path $P=y_{i_0} u_2 v_2 u_3 z_{j_0}$ in $G-u_0$
of length four joining $y_{i_0}$ with $z_{j_0}$.
Since $d(y_{i_0})=d(v_2)=2$, $d(u_2)=b+2$, and $d(u_3)=c+2$, we have

\begin{align*}
\alpha(y_{i_0}u_2) &\leq p+d(y_{i_0})+k-1=p+k+1, \\
\alpha(u_2v_2)  & \leq p+k+1+d(u_2)+k-1=p+2k+2+b,  \\
\alpha(v_2u_3)  &\leq p+2k+2+b+d(v_2)+k-1=p+3k+3+b,  \\
\alpha(u_3z_{j_0}) & \leq p+3k+3+b+d(u_3)+k-1=p+4k+4+b+c.  
\end{align*}

On the other hand, since $d(z_{j_0})=2$, we have

\begin{center}
$p+a+b+c-1\leq \alpha(u_0z_{j_0})=\overline{\alpha}(u_0) \leq
p+4k+4+b+c+d(z_{j_0})+k-1=p+5k+5+b+c$.
\end{center}
Hence, $a\leq 5k+6$, which is a contradiction.

\bigskip

Now we prove part (ii).
Let $G=\Delta_{r,s,t}$.
Suppose, to the contrary, that the graph $G$ has local deficiency at most $k$,
and consider a $q$-coloring $\beta$ realizing this property. Then $q\geq r+s+t$.

Consider the vertex $u_0\in V(G)$ with degree $\Delta(G)=r+s+t$. 
Let us consider two vertices $v$ and $w$ which are adjacent to the vertex $u_0$ and
$\beta(u_0v)=\underline{\beta}(u_0)$ and $\beta(u_0w)=\overline{\beta}(u_0)$. Let $p=\underline{\beta}(u_0)$.
By the definition of $\beta$, we have
$$p+r+s+t-1\leq \overline{\beta}(u_0)\leq p+r+s+t+k-1.$$
Consider the path $P=v u_i w$ in $G-u_0$ of length two joining
$v$ with $w$.
Since $d(v)=3$ and $d(u_i)\leq s+t$, we have
$$\beta(v u_i) \leq p+d(v)+k-1=p+k+2$$ 
and thus
$$\beta(u_i w)\leq p+k+2+ d(u_i)+k-1=p+2k+1+s+t.$$
On the other hand, since $d(w)=3$, we have
$$p+r+s+t-1\leq \beta(vw)=\overline{\beta}(v)\leq
p+2k+1+s+t+d(w)+k-1=p+3k+3+s+t.$$
Hence, $r\leq 3k+4$, which is a contradiction.
\end{proof}

Next, we have the following regarding weak local deficiency.

\begin{proposition}
\label{prop:rosettes2}
	The graphs $S_{a,a,a}$ and $\Delta_{a,a,a}$ have weak local deficiency $1$.
\end{proposition}
\begin{proof}
	Let us first consider the Sevastjanov rosettes. 

	We color the $a$ $2$-edge paths from $u_0$ to $u_1$ by colors $2i-1, 2i+1$, $i=1,\dots, a$ 
	(smallest color on the first edge of the path throughout),
	and similarly the $a$ $2$-edge paths from $u_0$ to $u_2$ by colors $2i, 2i+2$, $i=1,\dots,a$.
	Next, we color the $a$ $2$-edge paths from $u_0$ to $u_3$ by colors $2a+i, 2a+i+2$, $i=1,\dots,a$.
	Note that this yields an edge coloring that is interval at $u_0$, and weakly near-interval at all 
	$x_i, y_i$ and $z_i$.

	Finally, we color the edges on the path $v_{2} u_2 v_{1} u_1 v_{3}$ alternately by colors $2a+3, 2a+4$ 
	(starting with $2a+3$),
	and then $u_3 v_{2}$ by $2a+1$ and $u_3 v_{3}$ by $2a+2$. This yields a weak near-interval coloring of $S_{a,a,a}$.

\bigskip

	Next, we consider the graph $\Delta_{a,a,a}$.

	We shall sequentially color the edges incident to the vertices of degree $3$ in $G$. First
	we color the edges incident with vertices adjacent to $u_0, u_1, u_2$. 
	We color the edges incident with such vertices
	by colors $2i-1, 2i+1, 2i+3$, $i=1,\dots, a$, so that the edge incident with $u_0$ is colored $2i-1$, 
	the one incident 
	with $u_1$
	is colored $2i+1$, and the one incident with $u_2$ is colored $2i+3$. 
	Note that this yields a partial edge coloring that is
	weakly near-interval at every vertex.

	Next, we color the edges incident with vertices adjacent to $u_0, u_1, u_3$. 
	These edges are colored by $2i, 2i+2, 2i+4$, $i=1,\dots, a$,
	so that the edge incident with $u_0$ is colored $2i$, the one incident with 
	$u_1$ is colored $2i+2$, and the one incident with
	$u_3$ is colored $2i+4$. Note that this yields a partial edge coloring that is 
	weakly near-interval at every vertex.

	Third, we color the edges incident with vertices adjacent to 
	$u_0, u_2, u_3$. These edges are colored $2a+2i-1, 2a+2i, 2a+2i+1$
	so that the edge incident with $u_0$ is colored $2a+2i-1$, 
	the one incident with $u_2$ is colored $2a+2i$, and the one
	incident with $u_3$ is colored $2a+2i+1$. This yields a proper edge coloring that is 
	weakly near-interval at every vertex of 
	$\Delta_{a,a,a}$.
\end{proof}

From the preceding two propositions, we deduce the following.

\begin{corollary}
	For every positive integer $k$, there is a graph $G_k$ such that $\mathrm{def}_{loc}(G) - \mathrm{def}_{wloc}(G) \geq k$.
\end{corollary}
	
	Next, we consider the so-called generalized Hertz graphs, which was first considered in \cite{PetrosHrant}.

Let $T$ be a tree, $\mathcal{P}$ be the set of all paths in $T$ and $F(T)$ the set of all leaves in $T$,
We define the set $M(T)$ setting
\begin{equation*}
M(T)=\max_{P \in \mathcal{P}} \{|E(P)| + |\{uw \in E(T) : u \in V(P),
w \notin V(P)\}|\}.
\end{equation*}
Thus,  $M(T)$ is the maximum
number of edges with at least one endpoint in a single path in $T$.

Now let us define the graph $\widetilde{T}$ from $T$ and a new vertex $u$ as follows:
$$V(\widetilde{T})=V(T)\cup \{u\},
E(\widetilde{T})=E(T)\cup \{uv:v\in F(T)\}.$$
Clearly, $\widetilde{T}$ is a connected graph with
$\Delta(\widetilde{T})=\vert F(T)\vert$. Moreover, if $T$ is a tree
in which the distance between any two leaves is even, then
$\widetilde{T}$ is a connected bipartite graph.

We denote by $\mathrm{diam}(G)$ the diameter of a graph $G$. For generalized Hertz graphs,
we have the following.

\begin{theorem}
\label{trees} For any nonnegative integer $k$, if $T$ is a tree and $\vert F(T)\vert >
k(\mathrm{diam}(T)+1)+M(T)+2$, then $\mathrm{def}_{loc}(\widetilde{T})\geq k+1$.
\end{theorem}

\begin{proof}
Suppose, to the contrary, that $\widetilde{T}$ has local deficiency at most $k$, and
let $\alpha$ be proper $t$-coloring realizing this property, where $t\geq \vert F(T)\vert$.

Consider the vertex $u$. Let $v$ and $v^{\prime}$ be two vertices
adjacent to $u$ such that $\alpha(uv)=\underline{\alpha}(u)=s$ and
$\alpha(uv^{\prime})=\overline{\alpha}(u)\geq s+\vert F(T)\vert-1$. Since $\widetilde{T}-u$ is a tree, there is a unique path
$P(v,v^{\prime})$ in $\widetilde{T}-u$ joining $v$ with
$v^{\prime}$. Set
$$P(v,v^{\prime})=x_{1}x_{2}\ldots x_{i} x_{i+1}\ldots
x_{r}  x_{r+1},$$ 
where $x_{1}=v$, $x_{r+1}=v^{\prime}$.

Note that
\begin{center}
$\alpha(x_{i}x_{i+1})\leq s+1+\underset{j=1}{\overset{i}{\sum
}}(d_{T}(x_{j})+k-1)$ for $1\leq i\leq r$.
\end{center}

From this, we have

$$\alpha(x_{r}x_{r+1})=\alpha(x_{r}v^{\prime})\leq
s+1+\underset{j=1}{\overset{r}{\sum
}}(d_{T}(x_{j})+k-1)\leq s+M(T)+kr.$$

Hence

$$s+\vert F(T)\vert-1\leq \overline{\alpha}(u)=\alpha(uv^{\prime})\leq
s+M(T)+kr+k+1=s+1+k(r+1)+M(T),$$ 
taking into account that $r\leq \mathrm{diam}(T)$, we obtain 
$$\vert F(T)\vert\leq k(\mathrm{diam}(T)+1)+M(T)+2,$$
which is a contradiction. 
\end{proof}

As mentioned above,
our constructions with trees generalize the so-called
Hertz's graphs $H_{p,q}$, first described in
\cite{GiaroKubaleMalaf1}, which can be defined as
follows:
$$V(H_{p,q})=\left\{a,b_{1},b_{2},\ldots,b_{p},d\right\}\cup
\left\{c_{j}^{(i)}:1\leq i\leq p, 1\leq j\leq q\right\}$$ and
$$E(H_{p,q})=E_{1}\cup E_{2}\cup E_{3},$$
where
\begin{align*}
E_{1}= & \left\{ab_{i}:1\leq i\leq p\right\}, \\
E_{2}= & \left\{b_{i}c_{j}^{(i)}:1\leq i\leq p, 1\leq j\leq  
q\right\}, \\
E_{3}= & \left\{c_{j}^{(i)}d:1\leq i\leq p, 1\leq j\leq q
\right\}.
\end{align*}
Hertz's graphs are bipartite and has maximum degree $pq$.
Thus we have the following consequence of the preceding theorem.

\begin{corollary}
\label{Hertz graphs} For any positive integer $p,q\geq 2$ and  nonnegative integer $k$, if $pq>5k+p+2q+2$, then $\mathrm{def}_{loc}(H_{p,q})\geq k+1$.
\end{corollary}
Thus there are (bipartite) graphs $G$ with local deficiency $\Omega(\Delta(G) /5)$ and $\Omega(|V(G)|/5)$.

Next, we investigate the weak local deficiency of generalized Hertz graphs.

\begin{theorem}
\label{treess} For any nonnegative integer $k$, if $T$ is a tree and $\vert F(T)\vert >
(k+1)(M(T)+1)+1$, then $\mathrm{def}_{wloc}(\widetilde{T})\geq k+1$.
\end{theorem}
\begin{proof}
We proceed as in the proof of the preceding theorem.
Suppose, to the contrary, that $\widetilde{T}$ has weak local deficiency at most $k$, and
let $\alpha$ be proper $t$-coloring realizing this property, where $t\geq \vert F(T)\vert$.

As above, let $v$ and $v^{\prime}$ be two vertices adjacent to $u$ such that 
$\alpha(uv)=\underline{\alpha}(u)=s$ and
$\alpha(uv^{\prime})=\overline{\alpha}(u)\geq s+\vert F(T)\vert-1$. Consider the unique path
$P(v,v^{\prime})$ in $\widetilde{T}-u$ joining $v$ with
$v^{\prime}$. We set
$$P(v,v^{'})=x_{1}x_{2}\ldots x_{i}x_{i+1}\ldots
x_{r} x_{r+1},$$ 
where $x_{1}=v$, $x_{r+1}=v^{\prime}$.
Then, by the choice of $\alpha$,
\begin{center}
$\alpha(x_{i}x_{i+1})\leq s+k+1+\underset{j=1}{\overset{i}{\sum
}}(k+1)(d_{T}(x_{j})-1)$ for $1\leq i\leq r$.
\end{center}
So
$$\alpha(x_{r}x_{r+1})=\alpha(x_{r}v^{\prime})\leq
s+k+1+\underset{j=1}{\overset{r}{\sum
}}(k+1)(d_{T}(x_{j})-1)\leq s+(k+1)M(T).$$
Hence
$$s+\vert F(T)\vert-1\leq \overline{\alpha}(u)=\alpha(uv^{\prime})\leq
s+(k+1)M(T)+k+1=s+(k+1)(M(T)+1),$$ 
and so,
$$\vert F(T)\vert\leq (k+1)(M(T)+1)+1,$$
which is a contradiction. 
\end{proof}


Consider the graph $H_{p,q}$ defined above.  
We have that $\Delta (H_{p,q})=pq$ and $\vert V(H_{p,q})\vert=pq+p+2$. For the tree $T=H_{p,q}-d$, we have that $M(T)=p+2q$ and
$\vert F(T)\vert=pq$. From Theorem \ref{treess}, we can thus deduce the following.

\begin{corollary}
\label{Hertz graphss} For any positive integers $p,q\geq 2$ and nonnegative 
integer $k$, if $pq>(k+1)(p+2q+1)+1$, then $\mathrm{def}_{wloc}(H_{p,q})\geq k+1$.
\end{corollary}

This shows that there are families of graphs whose weak local deficiency grows with the number of vertices 
as well as with the maximum degree.

\section{Concluding remarks}

In this paper, we have introduced and studied the local and weak local deficiency of graphs.
We have shown that some families of graphs have small weak local deficiency, and given
several examples of graphs with large (weak) local deficiency.
In particular, there are families of graphs with local deficiency at least $\Omega(\Delta(G)/5)$.
Thus, the following question is natural.

\begin{problem}
	What is the order of growth of the maximum (weak) local deficiency 
	of graphs of maximum degree $\Delta(G)$?
\end{problem}

Similarly, we showed that there are families of graphs with local deficiency at least $\Omega(|V(G)|/5)$, so we could
ask the same question in terms of order of a graph.

We also showed that the difference between the deficiency, local deficiency and weak local deficiency can be
arbitrarily large. However, the following remains open.

\begin{problem}
	For any three positive integers $k >l > m$, is there a graph $G_{k,l,m}$
	such that $\defi(G_{k,l,m}) =k$,  $\defi_{loc}(G_{k,l,m}) = l$,
	$\defi_{wloc}(G_{k,l,m}) = m$?
\end{problem}

\bigskip

 In terms of maximum degree and the number of vertices, the smallest examples 
of graphs with no interval colorings have maximum degree $11$ and $19$ vertices, respectively.
Moreover, it was proved in \cite{CasselgrenPetrosyan}, that there is a graph with maximum degree $18$ with local deficiency at least $1$.

Consider the Hertz graph $H_{7,6}$ with $\Delta(H_{7,6})=42$ and 
$|V(H_{7,6})|=51$. From Corollary \ref{Hertz graphss}, we have that $H_{7,6}$ has no weak near-interval coloring.
This is the smallest example of a graph with weak local deficiency at least $1$ that we know of. 
We expect that there is a smaller example,
both in terms of maximum degree and number of vertices.

\begin{problem}
	What is the smallest graph, in terms of order or maximum degree, that has no weak near-interval coloring?
\end{problem}

More generally, we are interested in the following question.

\begin{problem}
\begin{itemize}
	\item[(i)]	What is the smallest graph, in terms of order or maximum degree, 
	with weak local deficiency at least $k$?

	\item[(ii)] What is the smallest graph, in terms of order or maximum degree,
	with local deficiency $k$?
\end{itemize}
\end{problem}

As regards constructive results, a natural continuation of the results in this paper would be to show that
graphs with maximum degree $5$ admit weak near-interval colorings; although this appears to be
somewhat difficult, we believe it to be a tractable problem. Similarly, proving that bipartite graphs
with maximum degree $8$ admit weak near-interval colorings would be a natural aim given the results
we obtained in this paper. 

\begin{conjecture}
$\,$
\begin{itemize}

\item[(i)] Graphs with maximum degree at most $5$ have weak local deficiency at most $1$.

\item[(ii)] Graphs $G$ with $\Delta(G)-\delta(G)\leq 2$ have weak local deficiency at most $1$.

\item[(iii)] Bipartite graphs with maximum degree at most $8$ have weak local deficiency at most $1$.

\end{itemize}

\end{conjecture}

As for complete multipartite graphs, there are known examples
of complete multipartite graphs with local deficiency $1$ (indeed, any such graph of Class 2), 
but we do not know whether there
are such examples with (weak) local deficiency $2$.

\begin{problem}
Is there a complete multipartite graph with (weak) local deficiency $2$?
\end{problem}

\end{document}